\documentclass[a4paper,12pt]{article}

\usepackage{amsmath}
\usepackage{amsthm}
\usepackage{amsfonts}
\usepackage{subeqnarray}
\usepackage{mathrsfs}
\usepackage{latexsym}

\usepackage[margin=2cm]{geometry}

\newcommand{\beq}{\begin{equation}}
\newcommand{\eeq}{\end{equation}expresion }
\newcommand{\ba}{\begin{eqnarray}}
\newcommand{\ea}{\end{eqnarray}}

\newtheorem{theorem}{Theorem}[section]
\newtheorem{remark}{Remark}[section]

\newtheorem{lemma}{Lemma}[section]
\newtheorem{corollary}{Corollary}[section]
\newtheorem{definition}{Definition}[section]

\title{Existence and uniqueness of a density 
probability solution for the stationary 
Doi-Edwards equation.}

\author{Ionel Sorin Ciuperca$^1$, Arnaud Heibig$^2$
\thanks{Corresponding authors.  E-mail: ciuperca@math.univ-lyon.fr and 
arnaud.heibig@insa-lyon.fr}   }

\numberwithin{equation}{section}
\begin{document}

\maketitle

\begin{flushleft}

Universit\'e de Lyon, CNRS, Institut Camille Jordan UMR 5208\\

$^1$ Universit\'e Lyon 1,  B\^at Braconnier, 43 Boulevard du 11 Novembre 1918, 
F-69622, Villeurbanne, France. 

$^2$ INSA-Lyon,  P\^ole de Math\'ematiques, B\^at. Leonard de Vinci No. 401, 
21 Avenue Jean Capelle, F-69621, Villeurbanne, France.

\end{flushleft}

\begin{abstract}

We prove the existence, uniqueness and non negativity of solutions
for a nonlinear stationary Doi-Edwards equation. The existence is proved
by a perturbation argument.
We get the uniqueness and the non negativity by showing
the convergence in time of 
the solution of the evolutionary Doi-Edwards equation 
towards any stationary solution. 

\end{abstract} 

\begin{flushleft}

Keywords: Polymeric fluids; Doi-Edwards equation; Stationary equation;
Well-posedness.
\end{flushleft}

\section{Introduction}\label{1}
It is well established that the modelling of non-Newtonian and viscoelastic 
flows bases on molecular theories.
In such theories, kinetical concepts are used to obtain
a mathematical description of the configuration of 
polymer chains. 
One of the most popular theories used to predict the behaviour of the melted
polymers is that of Doi and Edwards (see for exemple \cite{de78} and \cite{de89}).
It makes use of de Gennes reptation concept (\cite{dg79}).
In the Doi-Edwards model, chains of polymer
are confined within a tube of surrounding chains,
and chains can not
move freely. This description of the 
entanglement phenomenon leads to the 
concept of a primitive chain (the tube centerline).
The primitive chain, is not the real chain, and is shorter. Nevertheless, 
the
goal of Doi-Edwards theory is to 
describe the dynamic of the primitive chain. Basically, 
short time fluctuations of the polymer chain happen near the 
primitive chain in a wriggling motion, 
while fluctuations on larger time scales 
(say $t\geq T_{equilibration}$, see \cite{DE}) account for the 
chain ability to move inside the tube  (roughly speaking,
$T_{equilibration}$ is the time after which the primitive chain
feels the constraints imposed by the tube). This is the 
"snakelike" diffusive motion. Since 
diffusion  concerns the primitive chain, the
primitive chain finally disengages from the original tube.
This is a major complication in the theory, and for more details the reader is 
refered to \cite{DE}, \cite{de78} and \cite{de89}.
Nevertheless, notice that in the average (say on $\Delta t = T_{equilibration}$) 
the primitive chain and the real chain coincide.
Finally, for details on the thermodynamics of the model, see for  
instance 
\cite{de89}, \cite{pal04}.

From a mathematical point of view,  
a  primitive chain is represented as a curve
in $\mathbb{R}^3$. The position on the primitive chain is given by
a curvilinear coordinate $s \in [0, 1]$ (from now on, all the primitive chains are supposed
to have the same length which is normalized to 1).
Moreover, the orientation for any $s$ is given by a unitary vector $u$ tangent to
the curve; we then have $u \in S_2$ where $S_2$ is the unit sphere in
$\mathbb{R}^3$, that is:
$$
S_2 = \{ u \in \mathbb{R}^3, \quad \|u\| = 1 \}
   $$
where $\|\cdot\|$ is the Euclidian norm in $\mathbb{R}^3$.
The tangent vector $(s, u)$ is the microscopic variable of the model.

The rheology of such a fluid is obtained with the help of the so called 
{\it configurational probability density} of the molecules, denoted here by
$F$. It is a probability density with respect to the variable $u$.
Assuming space independence, we have \
$ F = F(t, s, u) $ \ where \ $t \geq 0$ is the time variable.
In the general case $ F = F(t, x, s, u) $ and one should 
write  equation \eqref{1:evol-princ} below with a convective term, i.e
replace $\partial F/ \partial t $ by 
the material derivative 
$\partial F/ \partial t + v\cdot\nabla_x F$. It would lead to  serious complications
since, in that case, a complementary equation (conservation law) is required
to determine $v$.
Here, as usual, $v$ stands for the macroscopic speed of the fluid.

The probability density satisfies the following PDE,  known under the name of
{\it Doi-Edwards equation}, and which is of Fokker-Planck-Smoluchowski type:
\begin{equation}
\frac {\partial F} {\partial t} - D \frac{\partial^2F}{\partial s^2}+  
\frac{\partial}{\partial u}.(\mathcal{G}F)
-\epsilon F\kappa : u\otimes u
+\epsilon 
\frac{\partial}{\partial s}\Big( F\kappa:\lambda(F)\Big)
=0  \textrm{ on  } S_2 \times ]0, 1[ 
\label{1:evol-princ}
\end{equation}
The ends of the chains are random, hence:
\begin{equation}
F(s=0) = F(s=1) = (1/4\pi)
\label{1:evol-cl}
\end{equation}
and for the initial condition: 
%
\begin{equation}
F(t=0) = F_0(s, u) 
\label{1:init}
\end{equation}
(see \cite{de89}, \cite{pal04} and \cite{CHP1}).

In the equation $\eqref{1:evol-princ} \; D > 0 $ and $\epsilon \geq 0$ are physical
coefficients and $\kappa = \kappa(t) \in \mathcal{M}_3(\mathbb{R})$ is the
velocity gradient; we also have
$$
\mathcal{G} =
\kappa.u - (\kappa:u \otimes u)u
  $$
and
$$ 
\lambda(F)(s) = \int_0^s\int_{S_2}
F(s', u)u \otimes ud\mu d s'.
   $$  
The case $\epsilon = 0$ corresponds to 
the so called {\it Independent Alignment Approximation} (IAA) 
for which explicit solutions 
of the evolutionary configurational PDE are known (see \cite{de78}).
In the case $\epsilon > 0$, the two mechanism
described by the terms 
$-\epsilon F\kappa : u\otimes u$
and $\epsilon 
\frac{\partial}{\partial s}\Big( F\kappa:\lambda(F)\Big)$
compensate, keeping constant the number of segments
by unit length:
$$\int_0^1\int_{S_2}\Big[
-\epsilon F\kappa : u\otimes u
+\epsilon 
\frac{\partial}{\partial s}\Big( F\kappa:\lambda(F)\Big)\Big]
d\mu(u)ds = 0
$$
In the present paper, we will make little use
of this relation, but it is likely that a thorough
analysis of the stationary problem (i.e for large $\epsilon$) would appeal to such
cancellation property. Note also that this is an ''ad hoc`` compensation 
since these two terms arise from two different phenomena. The first one 
quantifies the creation of new segments, while the second one
is due to the extension-retractation mechanism by which the 
chain keeps constant its curvilinear length.

Existence, uniqueness, and regularity of solutions of
\eqref{1:evol-princ}, \eqref{1:evol-cl}, \eqref{1:init} are proved
in \cite{CHP1}, as well as
the fact that $F$ is a
probability density. For existence
results in the case of related - but different - molecular models, see 
\cite{mas08}
\cite{CM}, \cite{JLB1}.
As an aside, notice that the Doi-Edwards model should not be mixed up
with what is commonly called the {\it Doi model}
(see \cite{ot08}), this latter being used for dilute polymers.
In Doi theory, molecules are considered as rigid dumbells.

In this paper we focus on the following stationary problem associated with
\eqref{1:evol-princ}, \eqref{1:evol-cl}:
\begin{align}
& - \frac{\partial^2F}{\partial s^2}+  
\frac{\partial}{\partial u}.(\mathcal{G}F)
-\epsilon F\kappa : u\otimes u
+\epsilon 
\frac{\partial}{\partial s}\Big( F\kappa:\lambda(F)\Big)
=0  \textrm{ on  } S_2 \times ]0, 1[
\label{KKU} 
\\
& F(s=0) = F(s=1) = (1/4\pi)
\label{YYUU}
\end{align}
In equation $\eqref{KKU}$, we set $D = 1$, which is not
restrictive, and we assume that the 
tensor $\kappa$ does not depend 
on $t$. Notice that stationary Fokker-Planck
equations with degenerate constitutive functions,  
but elliptic principal part, are studied for exemple in \cite{chupin1}, 
\cite{chupin2} and \cite{CIUPAL}.


The two points that are adressed in the sequel
are the well posedness and the non negativity of solutions of  
equations (\ref{KKU})-(\ref{YYUU}) (remark that,  in contrast with 
$F \geq 0$, equality
$\int_{S_2}F(u)d\mu(u) = 1$ can easily be obtained 
by integrating \eqref{KKU} on $S_2$
and making use of \eqref{YYUU}).
We will essentially
restrict 
to $|\epsilon|$ small, since global estimates 
on the sphere $S_2$ do not seem easy to obtain for 
$|\epsilon|$ large. As a matter of fact, even for
$\epsilon = 0$, well posedness of the stationary problem 
%
may not be obvious due to the lack of ellipticity
in the $u$-variable. 
Moreover, due to the probabilistic features of the equations,
the problem has to be well posed in 
$L^1(S_2)$, with some extra smoothness due for instance
to the $FLog(F)$ entropy estimates on the associated
time dependent problem (see for instance \cite{CM}). But $L^2(S_2)$ estimates 
are not expected.
Anyhow, proceeding as in 
\cite{de78} 
i.e writing
$$
f(s, u) := F(s, u)-\dfrac{1}{4\pi}= \sum_{n\in\mathbb{N}^*}
f_n(u)sin(n\pi s)
$$
the original problem \eqref{KKU}-\eqref{YYUU}  with $\epsilon = 0$
is reduced to a set of well posed problems 
in $L^r(S_2)$ with $r-1 \geq 0$ small enough (see section \ref{3}) :
\begin{align}
\dfrac{\partial}{\partial u}\cdot
\Big(\mathcal{G}f_n\Big)+n^2\pi^2f_n = g_n,\phantom{HHH} n\in\mathbb{N}^*\label{FRE}
\end{align}

Therefore, in order to prove existence for system \eqref{KKU}-\eqref{YYUU}, 
we proceed in
the following way. We first establish (Section \ref{3}) the 
existence and uniqueness 
for $\epsilon = 0$, and then, 
prove in Section \ref{4} the existence result for
$|\epsilon|$ small via the implicit function theorem.
Of course, a suitable fixed point procedure, using a 
variable basis of 
diagonalization for a Sturm-Liouville problem 
associated with
$\eqref{KKU}$, would also provide the general existence result, but at the cost of 
tedious estimates and notations. 
The advantage of the present approach is to work for $\epsilon = 0$ in a
fix Hilbertian basis of eingenvectors, 
namely $\big( \sqrt{2} \, sin(n\pi s) \big)_{n \in \mathbb{N}*}$,
and to extend the existence result by a transversality 
argument, supplying at the same quite strong
$L^r(S_2)$ estimates frequency by frequency.
To be thorough, remark that for $\epsilon = 0$, 
one can choose large  $r\geq 2$ for high
frequencies - but this is  not the case for low frequencies. As a 
consequence, solutions of the problem 
\eqref{KKU},\eqref{YYUU}
are obtained in a subspace of 
$W^{1, \infty}\big(0, T, L^r(S_2)\big)$, subspace
which is not easily  characterized in term of the classical
functional spaces. 
The restriction $r < 2$ on low frequencies also causes some difficulties
in the proof of the positivity of $F$.

Variants of the above arguments could be used to show uniqueness
of solutions of problem $\eqref{KKU}-\eqref{YYUU}$ by duality. 
Nevertheless, 
we shall obtain this result as a consequence of the proof that
$F$ is a  probability density. In order to prove this last result, 
we establish that solutions of problem $\eqref{KKU}-\eqref{YYUU}$
are the limits when $t\rightarrow \infty$ of solutions of the
time dependent Doi Edwards
problem. Since the solutions of the Doi Edwards 
problem are known to be  probability densities (see \cite{CHP1}), 
this provides the result; 
this approach also provides the desired uniqueness
(see Section \ref{5} and Section \ref{6}).
The main difficulty in the 
proof is to bound  on $\mathbb{R}_t^+$ 
in a suitable norm nonlinear terms such as
$\dfrac{\partial}{\partial s}\Big( F\kappa:\lambda(F)\Big)$. 

\section{Presentation of the problem and of the 
main results.}\label{2}
 
Throughout this paper 
we write $Q = \, ]0, 1[ \times S_2$. Making use of the Riemannian 
metric induced by the canonical
inner product $.$ of $\mathbb{R}^3$, we can define the 
usual surfacic measure $d\mu$ (or $d\mu(u)$), the 
gradient $\frac{\partial}{\partial u}$ and the divergence
$\frac{\partial}{\partial u} \cdot$ operators on  $S_2$
(see \cite{AM}). Since 
$S_2$ is a Riemannian submanifold of $\mathbb{R}^3$, 
the gradient of 
a smooth scalar valued function 
$g:S_2\rightarrow \mathbb{R}$
can aternatively be defined as the 
following projection (see \cite{HP}):
$$ \dfrac{\partial}{\partial u} g = 
\nabla_u\tilde{g} - \big(\nabla_u \tilde{g} \cdot u\big)u$$
where $\tilde{g}$
is any smooth extension of $g$ in a neighborhood 
of $S_2$ in $\mathbb{R}^3$ 
and $\nabla_u$ is the usual gradient in $\mathbb{R}^3$. 
Similarly,
for any smooth vector valued vector field of $S_2$, identified
with
$X\in C^{1}(S_2, \mathbb{R}^3)$ with 
$X\cdot u = 0$, the divergence of $X$ can be defined as 
(see \cite{HP}): 
$$
\dfrac{\partial}{\partial u} \cdot X 
=
\nabla_u \cdot \tilde{X}
-
\tilde{X}' u \cdot u
$$
where $\tilde{X}$ is any smooth 
extension of $X$ in a neighborhood  
of $S_2$ in $\mathbb{R}^3$. Notation 
$\tilde{X}'$ stands for the usual 
Jacobian matrix of $\tilde{X}$.
In what follows, we will essentially 
use Stokes formula:

\begin{align}
\int_{S_2}X\cdot\frac{\partial g}{\partial u}d\mu = -
\int_{S_2}\Big(\frac{\partial}{\partial u} \cdot X\Big)gd\mu\label{STOKES1}
\end{align}
valid for any smooth functions $X:S_2\rightarrow \mathbb{R}^3$ with
$X \cdot u = 0 $, and $g:S_2\rightarrow \mathbb{R}$. 
In particular, for 
$g = 1$, we get:
\begin{align}
\int_{S_2}\Big(\frac{\partial}{\partial u} \cdot X\Big)d\mu = 0\label{STOKES2}
\end{align}
Formulas \eqref{STOKES1} and \eqref{STOKES2} will be used to neglect or discard 
terms coming from $ \dfrac{\partial}{\partial u}
\cdot\big(\mathcal{G}f \big)$.

Using the following change of unknown function 
$f = F-\dfrac{1}{4\pi}$ and making use of:
\begin{align}
&\dfrac{\partial}{\partial u}\cdot \mathcal{G} = -3 
\kappa : u \otimes u 
\label{DKU}\\
&\kappa : Id_3 = tr(\kappa) = 0
\label{TR}
\end{align}
problem 
$\eqref{KKU}$-$\eqref{YYUU}$ 
becomes a homogeneous one:

\begin{align}
-\frac{\partial^2f}{\partial s^2}+  
\frac{\partial}{\partial u}\cdot(\mathcal{G}f)
-\epsilon f\kappa : u\otimes u
+&\epsilon \frac{\partial}{\partial s}\Big( f\kappa:\lambda(f)\Big)
\nonumber\\
&+\dfrac{\epsilon}{4\pi}\int_{S_2}\kappa:v \otimes vf(s,v )d\mu(v)
= \dfrac{3+\epsilon}{4\pi}\kappa:u \otimes u  \textrm{ on } Q\label{EDPH}\\
f(s=0) =f(s=1) = 0\phantom{UYTUY}\label{CLH}
\end{align}
In the following, we use a Hilbertian
basis of eigenvectors of the Laplacian in 
$]0, 1[$ with Dirichlet boundary conditions. Namely, 
family $(H_n)_{n\in\mathbb{N}^{*}}$ is defined by 
$$H_n(s)=\sqrt{2}sin(n\pi s)$$
For any $g\in L^1(Q)$, $n\in\mathbb{N}^*$, we write $g_n(u) = \int_0^1g(s, u)H_n(s)ds$.
For any $r\geq 1$, we define the vector spaces $X_r$ by:
\begin{align}\label{GH}
X_r=\big\{g& \in W^{1, \infty}\big(0, 1, L^r(S_2)\big) 
\textrm{ such that for any } n\in\mathbb{N^*},\nonumber\\ 
&\mathcal{G} \cdot \dfrac{\partial g_n}{\partial u} \in L^r(S_2)
\textrm{ and }
\displaystyle\sup_{n\in\mathbb{N^{*}}}\Big(n^3\Vert g _n \Vert_{L^r(S_2)}
\Big) + 
\displaystyle\sup_{n\in\mathbb{N^*}}\Big(n\Big\Vert\mathcal{G} \cdot
\dfrac{\partial g_n}{\partial u} \Big\Vert_{L^r(S_2)}
\Big) < \infty
\big\}
\end{align}
We will see in section $3$ that $X_r$
is a Banach space when endowed with its natural 
norm $\Vert . \Vert_{X_r}$: 
\begin{align}\label{NORM}
\Vert g \Vert_{X_r} = 
\displaystyle\sup_{n\in\mathbb{N^*}}\Big(n^3\Vert g _n \Vert_{L^r(S_2)}
\Big) + 
\displaystyle\sup_{n\in\mathbb{N^*}}\Big(n\Big\Vert\mathcal{G} \cdot
\dfrac{\partial g_n}{\partial u} \Big\Vert_{L^r(S_2)}
\Big) 
\end{align}
Moreover, we shall prove that 
any  $g\in X_r$ satisfies 
the homogeneous condition (see remark $\ref{reu}$ below):
$$g(s=0) = g(s=1)=0$$
Remark also that if $r_2 \geq r_1 \geq 1 $ then $X_{r_2}$ is continuously
embedded in $X_{r_1}$.

\begin{remark} Definition of $X_r$ is a simple but useful step in our 
analysis. It is formally obtained by counting the powers of
$n$ in equation \eqref{EDPH}. Notice for instance that we write
$n^3\Vert g _n \Vert_{L^r(S_2)}$ in place of 
$n^2\Vert g _n \Vert_{L^r(S_2)}$ as a corresponding term
to $-\partial^2 f / \partial s^2 $. This gain of one power 
in definition of $X_r$ arises from the right hand-side of equation 
\eqref{EDPH}, 
which does not depend of the $s$ variable. Indeed, 
for any $n\in \mathbb{N}^*$:
$$
\vert\int_0^1  (\kappa:u \otimes u) 
sin(n\pi s)
ds\vert 
\leq C/n
$$
supplying one power of $n$.
\end{remark}

Before giving the weak formulation of 
equations $\eqref{EDPH}-\eqref{CLH}$  in the $X_r$
functional frame, notice that
$\textrm{ for any } g, h \in X_r  \; \textrm{ we have }
\lambda({h})\in W^{2, \infty}(0, 1) 
\textrm{ and } g \in W^{1, \infty}\big(0, 1, L^r(S_2)\big)$
which implies:
\begin{align}\label{PROD}
\textrm{ for any } g, h \in X_r, 
\dfrac{\partial}{\partial s}\big(g \kappa:\lambda(h)\big)
\textrm{ is well defined and belongs to } L^{\infty}
\big(0, 1, L^r(S_2)\big)
\end{align}
\begin{definition}\label{DEF}
We say that $f$ is a weak 
solution 
of $\eqref{EDPH}-\eqref{CLH}$ if 
$f$ belongs to $X_1$ and satisfies:
\begin{align}\label{WEAK}
\int_{Q}\Big[
\dfrac{\partial f}{\partial s}\dfrac{\partial\phi}{\partial s}
-& f\mathcal{G}\cdot\dfrac{\partial\phi}{\partial u}
-\epsilon f \kappa : u \otimes u\phi
+\epsilon \dfrac{\partial}{\partial s} 
\big[f\kappa:\lambda(f)\big]
\phi
+\dfrac{\epsilon}{4\pi} 
\int_{S_2}\kappa : v \otimes v fd\mu(v)\phi
\Big]dQ\nonumber\\
=&\dfrac{3+\epsilon}{4\pi}
\int_{Q}\kappa : u \otimes u \phi dQ \quad 
\forall \; \phi \in H_0^1\big(0, 1, H^2(S_2)\big)
\end{align}
\end{definition}
The main result of this paper is:
\begin{theorem}\label{TH}
There exist 
$\epsilon_0>0$ 
such that, for any $\epsilon \in ]-\epsilon_0, \; 
\epsilon_0[$, there exists a unique weak solution 
$f_{\epsilon}$ of equations $\eqref{EDPH}-\eqref{CLH}$.
Moreover:
\begin{itemize} 
\item 
There exists \ \ $  r > 1 $ \ \ such that
\quad $f_{\epsilon} \in X_r \quad \forall \; \epsilon \in \; ]-\epsilon_0, \; \epsilon_0[$
\quad \quad (regularity result)
\item 
$f_{\epsilon}+(1/4\pi)$ is a probability density on 
$S_2$. That is, for any $s\in]0, 1[$, 
we have
\begin{equation}
\label{dens-prob}
\Big(f_{\epsilon}+\dfrac{1}{4\pi}\Big)(s) \geq 0 \textrm { a.e in } 
u\in S_2 \textrm { and }
\int_{S_2}\Big(f_{\epsilon}+\dfrac{1}{4\pi}\Big)d\mu(u) = 1
\end{equation}
\end{itemize}
\end{theorem}
In the sequel, 
we often drop the index $\epsilon$ (or $r\geq 1$) in 
the notations. In particular, from now on, we write $f$
in place of $f_{\epsilon}$. 

The above theorem is proved in two steps.
In a first step, the existence of a solution is established
via the implicit function theorem.
The rest of the theorem is 
obtained by showing that solution 
$F$ of problem $\eqref{KKU}-\eqref{YYUU}$ is the limit for 
$t\rightarrow +\infty$ of a family of 
density probablities $\big ( F(t)\big)_{t \geq 0}$, 
namely, the solution of an evolutionary 
Doi-Edwards equation.
\section{The case $\epsilon$ = 0.}\label{3}
We  give results related to the functional spaces
used in this paper. The existence part of theorem 
$\ref{TH}$
for 
$\epsilon = 0$ will follow from a priori estimates in these 
spaces.

\begin{lemma}\label{COM}
For any $r\in [1, +\infty[$, 
$X_r$ is a Banach space, continuously embedded 
in $W^{1, \infty}\big(0, 1, L^r(S_2)\big)$.
Moreover, for any $\phi\in X_r$, we have:
\begin{align}
\phi(s, u) = \sum_{n = 1}^{\infty}\phi_n(u)H_n(s)\label{FOU}
\end{align}
with absolute convergence in $W^{1, \infty}\big(0, 1, L^r(S_2)\big)$.
\end{lemma}

\begin{proof}
It is clear that $\Vert.\Vert_{X_r}$ is a seminorm 
on the vectorial space $X_r$. The fact that 
$\Vert.\Vert_{X_r}$ is a norm will be a 
straightforward consequence of equality 
$\eqref{FOU}$.

Let $\phi\in X_r$ and $n\in\mathbb{N^*}$. Then:
\begin{align}
\Vert\phi_n 
H_n\Vert_{W^{1, \infty}\big(0, 1, L^r(S_2)\big)}
\leq & C_1(1+\pi n) \dfrac{\Vert \phi \Vert_{X_r}}{n^3}\nonumber\\
\leq & C_2 \dfrac{\Vert \phi \Vert_{X_r}}{n^2}\nonumber
\end{align}
It implies that 
$\sum_{n = 1}^{\infty}\phi_n(u)H_n(s)$
is absolutely convergent in 
$W^{1, \infty}\big(0, 1, L^r(S_2)\big)$.

Now, for any $\psi\in L^{r'}(S_2)$, $ r^{-1} + r'^{-1} = 1$ and 
$N\in \mathbb{N^*}$, we have: 

\begin{align}
&\int_0^1\int_{S_2}\Big[ \phi(s, u)-
\sum_{n = 1}^{\infty}\phi_n(u)H_n(s)\Big]
H_N(s)\psi(u)dsd\mu(u)\nonumber\\
=&\int_{S_2}\phi_N(u)\psi(u)d\mu(u)
-
\sum_{n = 1}^{\infty}
\int_{S_2}\phi_n(u)\psi(u)d\mu(u)\int_0^1H_n(s)H_N(s)ds
 = 0
\end{align}
due to the absolute convergence of 
$\sum_{n = 1}^{\infty}\phi_n(u)\psi(u)H_n(s)H_N(s)$ 
in $L^{\infty}\big(0, 1, L^1(S_2)\big)$. This proves 
$\eqref{FOU}$  and the fact that $X_r$ 
is  continuously embedded 
in $W^{1, \infty}\big(0, 1, L^r(S_2)\big)$.

It remains to prove the completness of the space 
$\big(X_r, \Vert.\Vert_{X_r}\big)$. Let 
$(\phi^{p})_{p\in \mathbb{N}}$ be 
a Cauchy sequence in 
$X_r$.{}
Since $X_r$ 
is  continuously embedded 
in $W^{1, \infty}\big(0, 1, L^r(S_2)\big)$, 
$(\phi^{p})_{p\in \mathbb{N}}$
is also a Cauchy sequence in $W^{1, \infty}\big(0, 1, L^r(S_2)\big)$.
We denote by $\phi$ its limit in $W^{1, \infty}\big(0, 1, L^r(S_2)\big)$.
For any $n\in\mathbb{N^*}$:
\begin{align}
\Vert \phi_n^p -\phi_n\Vert_{L^r(S_2)}
\leq \sqrt{2} \int_0^1
\Vert \phi^p - \phi\Vert_{L^r(S_2)}(s)ds
\rightarrow 0\phantom{..}when\phantom{..}p \rightarrow +\infty\nonumber
\end{align}
Hence, $\phi_n^p \rightarrow\phi_n$ in $L^r(S_2)$
when $p \rightarrow +\infty$,
uniformely in $n\in\mathbb{N^*}$.
We also have that
$\Big(\mathcal{G}\cdot\dfrac{\partial \phi_n^p }{\partial u}\Big
)_{p\in\mathbb{N^*}}$ 
is a Cauchy sequence
in $L^r(S_2)$, hence convergent in $L^r(S_2)$. By identification, 
we deduce that 
$\mathcal{G}\cdot\dfrac{\partial \phi_n }{\partial u}$ 
belongs to $L^r(S_2)$ and that 
$\mathcal{G}\cdot\dfrac{\partial \phi_n^p }{\partial u}
\rightarrow
\mathcal{G}\cdot\dfrac{\partial \phi_n }{\partial u}$ in $L^r(S_2)$
for $p\rightarrow +\infty$.
From the inequalities: $$\Vert\phi_n^p-\phi_n^q\Vert_{L^r(S_2)}
\leq (1/n^3)\Vert \phi^p-\phi^q\Vert_{X_r}$$ and 
$$\Big\Vert 
\mathcal{G}\cdot\dfrac{\partial \phi_n^p }{\partial u}
-
\mathcal{G}\cdot\dfrac{\partial \phi_n^q }{\partial u}
\Big\Vert_{L^r(S_2)}
\leq \dfrac{1}{n}\Vert \phi^p-\phi^q\Vert_{X_r}
$$
we classically deduce, taking $q\rightarrow+\infty$, 
that $\phi\in X_r$ and $\phi^p \rightarrow \phi$ in 
$X_r$ for $p \rightarrow +\infty$.
\end{proof}
\begin{remark}\label{reu}
Formula $\eqref{FOU}$ implies that 
$\phi(s=0)=\phi(s=1)=0$ 
for any $\phi \in X_r$.
\end{remark}
Let us define for 
any $r\geq 1$ the space:
$$
Z_r=\{ \phi \in L^r(S_2) \textrm{ such that }
\mathcal{G}\cdot\dfrac{\partial \phi}{\partial u}
\in L^r(S_2)
\}
$$
which is clearly a Banach space  with norm 
$$
\Vert \phi \Vert_{X_r}
= 
\Vert\phi \Vert_{L^r(S_2)}+\Big\Vert
\mathcal{G}\cdot\dfrac{\partial \phi}{\partial u} \Big\Vert_{L^r(S_2)}
$$
The space $Z_r$ will be used in the  existence proof  
for $\epsilon = 0$. In order to perform estimates in $Z_r$, 
we first establish a useful formula (lemma $\ref{POK}$). 
Since this formula shall also be used for the 
evolution
Doi Edwards equation, we add the variable $t$ in the 
statement. Notice also that  
lemma $\ref{POK}$ can
not be reduced locally to the case 
$\mathcal{G}_{local\phantom{k}chart} = Cst$ due to the zeros of 
$\mathcal{G}$ on $S_2$.

\begin{lemma}\label{POK}
For any $T>0$, $r\geq1$ and 
$\phi \in L^r(]0, T[\times S_2)$
with 
$\mathcal{G}\cdot\dfrac{\partial \phi}{\partial u} 
\in L^r(]0, T[\times S_2)$
we have

\begin{align}\label{QS}
r\vert \phi \vert^{r-1}sgn(\phi) \mathcal{G}
\cdot\dfrac{\partial \phi}{\partial u}  
=
\mathcal{G}
\cdot\dfrac{\partial }{\partial u}
(\vert\phi\vert^r)
\end{align}
\end{lemma}
\begin{proof}
Using local charts, this amounts essentially 
to prove that for any open bounded set 
$\Omega \subset \mathbb{R}^3$, 
$A\in C^{\infty}(\Omega, \mathbb{R}^3)$, 
$\psi \in  L^r(\Omega)$ 
with 
$A\cdot \nabla\psi \in L^r(\Omega)$, 
we have:
\begin{align}\label{FF}
r\vert \psi \vert^{r-1}sgn(\psi) A
\cdot
\nabla \psi 
=
A \cdot
\nabla
(\vert\psi\vert^r)
\end{align}
Let us consider a sequence 
$\big(\psi_n\big)_{n\in\mathbb{N}}$ 
in $C^{\infty}(\Omega)$
endowed with the two following properties 
(see Lemma II.1 of \cite{dpl89}):
\begin{align}
&\psi_n\rightarrow\psi \textrm{ in } L^r(\Omega) \textrm{ for } n\rightarrow +\infty \label{DP1}\\
&A\cdot\nabla\psi_n \rightarrow A\cdot\nabla\psi \textrm{ in } L^r(\Omega) 
\textrm{ for } n\rightarrow +\infty \label{DP2}
\end{align}
and,  for any $\delta>0$, define functions 
$h_{\delta}:\mathbb{R}\rightarrow \mathbb{R}$ and 
$j_{\delta}:\mathbb{R}\rightarrow \mathbb{R}$
by $h_{\delta}(y)=\sqrt{y^2+\delta}$ and 
$j_{\delta}(y)=y/\sqrt{y^2+\delta}$. We classically have: 
\begin{align}\label{KO}
r\big[h_{\delta}( \psi_n )\big]^{r-1}j_{\delta}(\psi_n ) A
\cdot \nabla\psi_n 
=
A
\cdot\nabla\Big[
\big(h_{\delta}(\psi_n)\big)^r\Big]
\end{align}
($\delta>0, n\in\mathbb{N^*}$).

We extract a subsequence of $\big(\psi_n\big)_{n\in\mathbb{N}^*}$, still denoted 
by $\big(\psi_n\big)_{n\in\mathbb{N}^*}$, such that:
\begin{align}
&\psi_n \rightarrow \psi \textrm{ a.e } 
\textrm{ for }n\rightarrow +\infty\label{KP}\\
&\vert\psi_n\vert \leq 
\psi^*  \textrm{ a.e, with } \psi^* \in L^r(\Omega)\label{KQ}
\end{align}
Fix $\delta>0$. Our goal is to pass to the limit when 
$n\rightarrow +\infty$ in $\eqref{KO}$. 
We have:
\begin{align}
\vert
h_{\delta}(\psi_n)
-
h_{\delta}(\psi)\vert
\leq&
\dfrac{\vert\psi_n-\psi\vert
\vert\psi_n+\psi\vert}
{\sqrt{\psi_n^2+\delta}+\sqrt{\psi^2+\delta}}\nonumber\\
\leq&
\vert\psi_n-\psi\vert\nonumber
\end{align}
Hence, from $\eqref{DP1}$, we get $\big[ h_{\delta}(\psi_n)\big]^r
\rightarrow \big[h_{\delta}(\psi)\big]^r$ for 
$n\rightarrow \infty$. It implies that, for $n 
\rightarrow +\infty$: 
\begin{align}
A\cdot\nabla\Big[
h_{\delta}(\psi_n)^r
\Big]
\rightarrow
A\cdot\nabla\Big[
h_{\delta}(\psi)^r
\Big] \textrm{ in }\mathscr{D}'(\Omega)\label{JAQ}
\end{align}
Observe that:
\begin{align}
\vert
j_{\delta}(\psi_n)\vert
\vert h_{\delta}(\psi_n)\vert^{r-1}
\leq& 
\big[(\psi^*) ^2+\delta\big]^{(r-1)/2}\nonumber\\
\leq&
\big[\psi^*+\delta\big]^{r-1}\in L^{r'}(\Omega)\label{MIM}
\end{align}
with $r'\in [1, +\infty]$
such that $r^{-1}+r'^{-1} = 1$.
For $r>1$, using the dominated convergence theorem, 
we deduce from $\eqref{KP}$ and $\eqref{MIM}$ that:
\begin{align}
j_{\delta}(\psi_n)h_{\delta}(\psi_n)^{r-1}
\rightarrow
j_{\delta}(\psi)h_{\delta}(\psi)^{r-1}\textrm{ in }
L^{r'}(\Omega) \textrm{ for } n\rightarrow +\infty\label{MOW}
\end{align}
Using $\eqref{DP2}$ and $\eqref{MOW}$, 
we conclude that, for $r>1$:
\begin{align}
j_{\delta}(\psi_n)h_{\delta}(\psi_n)^{r-1}
A\cdot\nabla\psi_n 
\rightarrow
j_{\delta}(\psi)h_{\delta}(\psi)^{r-1}
A\cdot\nabla\psi 
\textrm{ in }
L^{1}(\Omega) \textrm{ for } n\rightarrow +\infty\label{MEW}
\end{align}
For $r=1$, we easily obtain: 
\begin{align}
j_{\delta}(\psi_n)
A\cdot\nabla\psi_n 
\rightarrow
j_{\delta}(\psi)
A\cdot\nabla\psi 
\textrm{ in }
L^{1}(\Omega) \textrm{ for } n\rightarrow +\infty\label{MAW}
\end{align}
We deduce from $\eqref{KO}$, $\eqref{JAQ}$, $\eqref{MEW}$, $\eqref{MAW}$, 
that:
\begin{align}\label{KA}
r\big[h_{\delta}(\psi)\big]^{r-1}j_{\delta}(\psi) A
\cdot \nabla\psi
=
A
\cdot\nabla\Big[
\big(h_{\delta}(\psi)\big)^r\Big]
\end{align}
for $r \geq 1$.  In order to pas to the limit 
$\delta \rightarrow 0$ in the above equality, notice that: 
\begin{align}
\vert h_{\delta}(\psi)\vert
\leq \vert\psi\vert + 1 \textrm { for }
\delta \leq 1\label{ADA}
\end{align}
For $\delta \rightarrow 0$, we have $h_{\delta} \rightarrow \vert.\vert$ everywhere. 
Due to 
$\eqref{ADA}$, $\psi \in L^r(\Omega)$ and the dominated convergence 
theorem, we conclude that, for any $r \geq 1$:
\begin{align}
h_{\delta}(\psi)^r \rightarrow \vert\psi\vert^r 
\textrm{ in } L^1(\Omega)\textrm{ when }
\delta \rightarrow 0\label{BTJ}
\end{align}
for any $r \geq 1$. Arguing similarly, we also prove that:
\begin{align}
j_{\delta}(\psi)h_{\delta}(\psi)^{r-1}A.\nabla \psi \rightarrow
sgn(\psi)\vert\psi\vert^{r-1}
A.\nabla \psi
\textrm{ in } L^{1}(\Omega)\textrm{ when }
\delta \rightarrow 0\label{BTK}
\end{align}
Using $\eqref{BTJ}$, $\eqref{BTK}$ and $\eqref{KA}$, 
we obtain the result.
\end{proof}
Let us introduce for any $r \geq 1$  the space:
\begin{align}
Y_r =  \{
(a_n)_{n\in\mathbb{N}^*}
\textrm{ such that: } 
\forall
n\in\mathbb{N}^*,
a_n\in L^r(S_2) 
\textrm{ and }
\displaystyle\sup_{n\in\mathbb{N^{*}}}\Big(n\Vert
a_n
\Vert_{L^r(S_2)}\Big) < \infty
\} \nonumber
\end{align}
It is clear that $Y_r$,
when endowed with its natural norm: 
\begin{align}
\Vert(a_n)_{n\in\mathbb{N}^*}\Vert_{Y_r} = 
\displaystyle\sup_{n\in\mathbb{N^{*}}}\Big(n\Vert
a_n
\Vert_{L^r(S_2)}\Big)\nonumber
\end{align}
is a Banach space. Next, we introduce the linear, bounded 
operator $\mathscr{T}_0: X_r \rightarrow Y_r$ defined for any
$g\in X_r$ by $\mathscr{T}_0(g) = (a_n)_{n\in\mathbb{N}^*}$ with:

\begin{align}
a_n = n^2 \pi^2 g_n + 
\frac \partial {\partial u} \cdot (\mathcal{G} g_n) 
\nonumber
\end{align}

where we recall that
\begin{align}
g_n = \int_0^1g(s)H_n(s)ds\label{DYP}
\end{align}
This operator is formally obtained 
by projecting 
the left-hand side of equation $\eqref{EDPH}$ for 
$\epsilon = 0$ on the Hilbertian basis 
$(H_n)_{n\in\mathbb{N}^*}$ of $L^2(]0, 1[)$.

In order to study $\mathscr{T}_0$, 
we first introduce the following 
unbounded linear operator defined 
for any $n\in\mathbb{N}^*$ and 
$r>1$ by:$$L_n: 
L^r(S_2)\rightarrow L^r(S_2)$$
where $D(L_n) = Z_r$ and for  any $h \in Z_r$: 
$$ L_n(h)=n^2\pi^2h+\dfrac{\partial}{\partial u}\cdot(\mathcal{G}h)$$
It is clear that $L_n$ is closed and densely defined.
\\
Let us consider $ r' > 1 $ such that  
\begin{equation}
\label{r-prim}
\frac 1 r + \frac 1 {r'} = 1.
\end{equation}
One can easily prove that 
the adjoint operator 
$$L_n^{*}: 
L^{r'}(S_2)\rightarrow L^{r'}(S_2)$$ is such that
$D(L_n^{*}) = Z_{r'}$ and for any  $\psi \in Z_{r'}$: 
$$ L_n^{*}(\psi)=n^2\pi^2\psi-\mathcal{G}\cdot\dfrac{\partial \psi}{\partial u}$$
\begin{lemma}\label{ESP}
There exists $r_0>1$ such that: 
$L_{n}:  Z_r \rightarrow L^r(S_2)$ is a Banach isomorphism
for any $r\in]1, r_0[$ and $n\in\mathbb{N}^*$. 
Moreover, there exists $C>0$ such that 
\begin{align}
&\Vert L_n^{-1}(\psi) \Vert_{Z_r}\leq C \Vert \psi \Vert_{L^r(S_2)}\label{BV1}\\
&\Vert L_n^{-1}(\psi) \Vert_{L^r(S_2)}\leq 
\dfrac{C}{n^2}
\Vert\psi\Vert_{L^r(S_2)}\label{BV2}
\end{align}
for any $\psi \in L^r(S_2)$, $n\in\mathbb{N}^*$
and $r\in]1, r_0[$.
\end{lemma}
\begin{proof}
Let $r' > 1$ satisfying \eqref{r-prim}.
The surjectivity of $L_n$ is a 
consequence of the following a priori
estimate:
\begin{align}
\forall \varphi \in Z_{r'}, \Vert \varphi \Vert_{L^{r'}(S_2)}
\leq C_1 \Vert L_{n}^*(\varphi)\Vert_{L^{r'}(S_2)}\label{BBB}
\end{align}
In order to prove $\eqref{BBB}$, set $h = L_n^{*}(\varphi)$.
We have:
\begin{align}
n^2\pi^2\varphi - \mathcal{G} \cdot \dfrac{\partial \varphi}{\partial u} =h
\end{align}

We multiply this inequality by 
$\vert \varphi\vert^{r'-1}sgn(\varphi)$, integrate over 
$S_2$ and use lemma  $\ref{POK}$ and we get:
\begin{align}
n^2\pi^2\int_{S_2}\vert \varphi \vert^{r'} d\mu - \dfrac{1}{r'}
\int_{S_2} \mathcal{G}\cdot \dfrac{\partial 
(\vert \varphi^{r'} \vert)}{\partial u}d\mu = 
\int_{S_2} h\vert \varphi \vert^{r'-1}sgn(\varphi)d\mu
\end{align}
Using the Stokes formula and Holder inequality 
we obtain:
\begin{align}
\int_{S_2}\Big (n^2\pi^2-\dfrac{3}{r'}\kappa: u \otimes u \Big)
\vert \varphi \vert^{r'} d\mu 
\leq
\Vert h \Vert_{L^{r'}(S_2)} \Vert\varphi\Vert_{L^{r'}(S_2)}^{r'-1}\nonumber
\end{align}
Taking $r'$ large enough, that is $r-1$ small enough
we get 
$\eqref{BBB}$, which proves that $L_n$ is
surjective.

We now prove the injectivity of $L_n$.
Let us denote $\psi = L_n(g)$, with 
$g \in Z_r$, $\psi\in L^{r}(S_2)$. Hence:

\begin{align}
n^2\pi^2g + \dfrac{\partial}{\partial u}\cdot  
\big(\mathcal{G}g\big) = \psi\label{FNK}
\end{align}
Using lemma $\ref{POK}$, we get:
\begin{align}
\vert g \vert^{r-1} sgn(g)\dfrac{\partial}{\partial u}\cdot  
\big(\mathcal{G}g\big)
&=
\vert g \vert^r 
\dfrac{\partial}{\partial u}\cdot  
\mathcal{G}
+
\dfrac{1}{r}\mathcal{G}\cdot
\dfrac{\partial}{\partial u}
\big( \vert g \vert^r \big)\nonumber\\
&
=\dfrac{\partial}{\partial u}
\cdot 
\big(\mathcal{G}\vert g \vert^r  \big)
+
\big(
\dfrac{1}{r}-1
\big)
\mathcal{G}
\cdot
\dfrac{\partial}{\partial u}
\big(
\vert g\vert^r 
\big)
\label{XLM}
\end{align}
We now multiply 
$\eqref{FNK}$ by $\vert g \vert^{r-1} sgn(g)$
and integrate over
$S_2$ to get:
\begin{align}
\int_{S_2}\Big(n^2\pi^2 
-
\dfrac{3(r-1)}{r}\kappa:u\otimes u
\Big)
\vert g \vert^{r}d\mu
\leq 
\Vert \psi \Vert_{L^{r}(S_2)}\Vert g \Vert_{L^{r}(S_2)}^{r-1}
\end{align}
Taking again $r-1$ small enough we obtain 
at the same time that 
$L_n$ is one to one and estimate $\eqref{BV2}$. Estimate
$\eqref{BV1}$ follows from equality 
$$\mathcal{G}\cdot \dfrac{\partial g}{\partial u} 
= 
\psi -\big( n^2\pi^2 - 3\kappa:u\otimes u\big)g $$ 
(see eqs. $\eqref{FNK}$ and $\eqref{DKU}$) and estimate $\eqref{BV2}$.
\end{proof}
\begin{remark} 
\label{3:rem1}
In the above proof, we can choose $r\geq 2$
for large $n\in\mathbb{N}^*$.
\end{remark}

Since for any $g\in X_r$ we have 
$\big( \mathscr{T}_0(g)\big)_n = L_n(g_n)$
where $g_n$ is given by $\eqref{DYP}$, we easily obtain 
the following:
\begin{corollary}\label{ZER}
There exists $r_0>1$ such that for any 
$r\in \, ]1, r_0[$, $\mathscr{T}_0$ is a Banach isomorphism.
\end{corollary}
\section{Proof of the existence result for $\epsilon$ small.}\label{4}
For $\vert \epsilon \vert$ small enough, existence of solutions 
for the problem $\eqref{EDPH}$-$\eqref{CLH}$ will be 
a consequence of corollary $\ref{ZER}$ and the implicit function 
theorem for an 
appropriate operator 
$\mathscr{T}: \mathbb{R} \times X_r \rightarrow Y_r$.
In order to handle the nonlinearity of such an operator, 
we prove a preliminary lemma.
Notice first that for any $n\in\mathbb{N}^*$, due to remark 
$\ref{PROD}$:
\begin{align}
b_n = \int_0^1 \dfrac{\partial}{\partial s}\big[\phi \kappa:
\lambda(\psi)\big](s) H_n(s)ds\label{QTG}
\end{align}
is well defined and belongs to 
$L^r(S_2)$ for any $ \phi, \psi \in X_r$.
\begin{lemma}\label{OHM}
For any $r \geq 1$
let $B : X_r \times X_r \rightarrow Y_r$ 
be given by $B(\phi,\psi)=(b_n)_{n\in\mathbb{N}^*}$
where $b_n$ is given by $\eqref{QTG}$.

The function $B$ is well defined, bilinear 
and continuous.  Moreover, for any 
$\phi, \psi\in X_r$ and $r \geq 1$ we have: 
\begin{align}
\Vert b_n \Vert_{L^r(S_2)} \leq \dfrac{C}{n^2} \Vert \phi \Vert_{X_r}
\Vert \psi \Vert_{X_r}\label{ABF}
\end{align}
where $C>0$ is a constant.
\end{lemma}
\begin{proof}
In order to prove inequality $\eqref{ABF}$, 
we integrate by part equation $\eqref{QTG}$. We get:
\begin{align}
b_n = -\sqrt{2}n\pi\int_0^1\phi\kappa:\lambda(\psi)
cos(n\pi s)ds\label{XXX}
\end{align}
Hence, we just have to prove that: 
\begin{align}
\Vert\int_0^1 \phi \kappa:\lambda(\psi)
e^{-i n\pi s}ds\Vert_{L^{r}(S_2)}
\leq 
\dfrac{C}{ n^3}
\Vert \phi \Vert_{X_r}
\Vert \psi \Vert_{X_r}\label{ZZZ}
\end{align}
for any $n\in\mathbb{N}^*$, with $C>0$
independent of $n, \phi, \psi$.

Observe that:
\begin{align}
\kappa:\lambda(\psi)(s) =& \int_0^s\int_{S_2}
\Big[\kappa:v \otimes v 
\sum_{q=1}^{+\infty}\psi_q(v)H_q(s')\Big]dvds'\nonumber\\
=&\sqrt{2}\sum_{q=1}^{+\infty}\Big\{\dfrac{1}{q\pi}\big[1-cos(q \pi s)\big]
\int_{S_2}\psi_q(v)\kappa:v \otimes vdv\Big\}\label{NHN}
\end{align}
We have by definition of $\Vert.\Vert_{X_r}$
and Holder inequality:
\begin{align}
\big\vert
\int_{S_2}\psi_q(s)(v)\kappa:v \otimes v dv
\big\vert
\leq
\dfrac{C}{q^{3}}\Vert \psi \Vert_{X_r}\label{CLA}
\end{align}
Hence: 
\begin{align}
\sum_{q=1}^{+\infty}\Big\{\dfrac{1}{q}
\big\vert
\int_{S_2}\psi_q(v)\kappa:v \otimes vdv\big\vert\Big\}
\leq
C\Vert \psi \Vert_{X_r}\label{NHM}
\end{align}
It follows from $\eqref{NHN}$, $\eqref{CLA}$ and $\eqref{NHM}$ that:
\begin{align}
\kappa:\lambda(\psi)(s) =
\sum_{p\in\mathbb{Z}}^{}
\lambda_p e^{i p \pi s}
\end{align}
with:
\begin{align} 
&\bullet \textrm{ For } p=0 , \vert \lambda_0 \vert
\leq C \Vert \psi \Vert_{X_r}\label{B1}\\
&\bullet \textrm{ For }p\in\mathbb{Z}^*, \vert \lambda_p \vert
\leq \dfrac{C}{p^4} \Vert \psi \Vert_{X_r}\label{B2}
\end{align}
As a consequence, $\sum_{p\in\mathbb{Z}}^{}
\lambda_p e^{i p\pi s}$ is absolutely convergent in
$L^{\infty}(0, 1)$.

On the other hand, we can write:
\begin{align}
\phi(s) =& 
\sum_{q=1}^{\infty}
\phi_q sin(q\pi s)\nonumber\\
=&\sum_{q\in\mathbb{Z}}^{}
\tilde{\phi_q} e^{i q\pi s}\label{RWW}
\end{align}
with $\tilde{\phi_0} = 0$, $\tilde{\phi_q} = -(i/2) \phi_q$ for $q>0$ and 
$\tilde{\phi_{q}} = (i/2) \phi_{-q}$ for $q<0$. Hence, for any 
$q\in\mathbb{Z}^*$:
\begin{align}
\Vert \tilde{\phi_{q}} \Vert_{L^{r}(S_2)}
\leq&
\dfrac{1}{2}
\Vert \phi_{q} \Vert_{L^{r}(S_2)}\nonumber\\
\leq&\dfrac{1}{2 \vert q \vert^3}
\Vert \phi \Vert_{X_r}\label{A1}
\end{align}
It follows that
$\sum_{q\in\mathbb{Z}}^{}
\tilde{\phi_q} e^{i q\pi s}$ 
is absolutely convergent in 
$L^{\infty}(0, 1, {L^{r}(S_2)})$. Invoking 
a classical result of the product of absolutely convergent series 
in Banach spaces, we find that:
\begin{align}
\phi\kappa:\lambda(\psi)=
\sum_{n\in\mathbb{Z}}^{}
h_n e^{i n\pi s}\nonumber
\end{align}
with absolute convergence in $L^{\infty}(0,1,L^{r}(S_2))$.
Moreover, since, for any $n\in\mathbb{Z}$ we have:
\begin{align}
h_n = 
\sum_{q\in\mathbb{Z}}^{}
\lambda_{n-q}\tilde{\phi_q}
\end{align}
we can write, restricting to $n\in\mathbb{N}^*$ and 
making use of inequalities 
$\eqref{B1}$, $\eqref{B2}$, $\eqref{A1}$: 
\begin{align}
\Vert h_n\Vert_{L^{r}(S_2)}
\leq&
\vert \lambda_0 \vert
\Vert \tilde{\phi_n}\Vert_{L^{r}(S_2)}
+
\sum_{\vert q\vert \geq (n/2)}^{}
\vert \lambda_{n-q} \vert
\Vert \tilde{\phi_q}\Vert_{L^{r}(S_2)}
+
\sum_{0<\vert q \vert < (n/2)}^{}
\vert \lambda_{n-q} \vert
\Vert \tilde{\phi_q}\Vert_{L^{r}(S_2)}\nonumber\\
\leq&
\Vert \psi \Vert_{X_r}\Vert \phi \Vert_{X_r}
\Big[
\dfrac{C}{n^{3}}
+
C\big(
\sum_{k\in\mathbb{Z}^{*}}
\dfrac{1}{k^{4}}\big)
\big(\dfrac{2}{n}
\big)^{3}
+
C\big(
\sum_{q\in\mathbb{Z}^{*}}
\dfrac{1}{q^{3}}\big)
\big(\dfrac{2}{n}
\big)^{4}
\Big]\nonumber\\
\leq&
\dfrac{C}{n^{3}}
\Vert \phi \Vert_{X_r}\Vert \psi \Vert_{X_r}\nonumber
\end{align}
This implies $\eqref{ZZZ}$. Due to equality $\eqref{XXX}$, we finally 
get $\eqref{ABF}$.
\end{proof}
Let $r \geq 1$. We introduce the operator: 
$$
\mathscr{T}: \mathbb{R} \times X_r \rightarrow Y_r
$$
defined for any $\epsilon \in \mathbb{R}$ and 
$g \in X_r$ by $\mathscr{T}(\epsilon, g) = 
(d_n)_{n\in\mathbb{N}^*}$ with: 
\begin{align}
d_n = 
n^2\pi^2g_n
+
\dfrac \partial {\partial u} \cdot \big(\mathcal{G}g_n\big)
-
\epsilon \kappa:u \otimes u g_n
+&
\epsilon \big(B(g, g)
\big)_n\nonumber\\
+&
\dfrac{\epsilon}{4\pi}
\int_{S_2}\kappa:v \otimes v g_n(v) dv
-
\dfrac{3+\epsilon}{4\pi}\kappa:u \otimes u 1_n\label{YOP}
\end{align}
In this writing,  $g_n$ is given by 
$\eqref{DYP}$. Coefficient $1_n$ is such that
$1 = \sum_{n\in\mathbb{N}^*}^{} 1_n H_n(s) $, with convergence 
in $L^2(0, 1)$, that is:
\begin{align}
1_n=&\sqrt{2}\int_{S_2} sin(n\pi s) ds\nonumber\\
=&\dfrac{\sqrt{2}}{n\pi}[1-(-1)^n]\label{UGH}
\end{align}
We can formulate problem ($\ref{EDPH}, \ref{CLH}$) 
in term of operator 
$\mathscr{T}$:
\begin{lemma}\label{WEA}
Let $(\epsilon, f) \in\mathbb{R}\times  X_r$ with $r \geq 1$. 
Function $f$ is a weak solution of 
$\eqref{EDPH}-\eqref{CLH}$ if and only if 
$\mathscr{T}(\epsilon, f) = 0$.
\end{lemma}
\begin{proof}
Let $f\in X_r$
be a weak solution 
of $\eqref{EDPH}-\eqref{CLH}$.
Taking in 
$\eqref{WEAK}$ 
$\phi(s, u) = \psi(u) sin(n \pi s)$
with arbitrary $\psi\in H^2(S_2)$ 
and $n\in\mathbb{N}^*$, we obtain after integration 
by parts in $s$ that 
\\
$\mathscr{T}(\epsilon, f) = 0$.

Conversely, let us consider 
$f\in X_r$ such that 
$\big(\mathscr{T}(\epsilon, f)\big)_n = 0$ 
for any $n\in\mathbb{N}^*$.
This implies that for any 
$\phi \in H_0^1\big(0, 1, H^2(S_2)\big)$ and any 
$m\in\mathbb{N}^*$ we have:
\begin{align}
\int_{S_2}\int_{0}^1\Big[
-&\frac{\partial^2f^{(m)}}{\partial s^2}\phi-  
f^{(m)}\mathcal{G}\cdot\frac{\partial \phi}{\partial u}
-\epsilon f^{(m)}\kappa : u\otimes u \phi
+\epsilon h^{(m)} \phi
\nonumber\\
&+\dfrac{\epsilon}{4\pi}\int_{S_2}\kappa:v \otimes vf^{(m)}(v,s )d\mu(v)\phi
-\dfrac{3+\epsilon}{4\pi}\kappa:u \otimes u 1^{(m)}\phi\Big]
dsd\mu(u)= 0\label{BUD}
\end{align}
In the above equation, exponent 
$^{(m)}$ indicates a $L^2$ projection 
on $span(H_1, ..., H_m)$, i.e: $$f^{(m)}(s) = \sum_{n=1}^m 
\big[\int_0^1 f(\sigma)H_n(\sigma)d\sigma\big]H_n(s)$$,  
$$h^{(m)}(s) = \sum_{n=1}^m \Big[\int_0^1 
\dfrac{\partial}{\partial s}\big(f\kappa: \lambda(f)\big)(\sigma)H_n(\sigma)d\sigma\Big] 
H_n(s)$$, $$1^{(m)}(s) = \sum_{n=1}^m 
1_nH_n(s) $$
We integrate by parts with respect to the $s$ variable the first term of 
$\ref{BUD}$. Using convergences  
$f^{(m)} \rightarrow f$ in 
$W^{1, \infty}\big(0, 1, L^r(S_2)\big)$ (see $\ref{FOU}$)
and $h^{(m)} \rightarrow\dfrac{\partial}{\partial s}\big(f\kappa: \lambda(f)\big)$
in $L^{\infty}\big(0, 1, L^r(S_2)\big)$
(see $\ref{ABF}$), we obtain the result.
\end{proof}
We are in position to prove the existence 
and regularity part in theorem $\ref{TH}$:
\begin{proof}\label{KUI}
We consider $r \in \,  ]1, \, r_0[$ with $r_0 > 1$ given by lemma \ref{ESP}.
It is clear from lemma $\ref{OHM}$ that 
$\mathscr{T}$ is a $C^{\infty}$
function. Remark that, for any $g\in X_r$, 
$\mathscr{T}(0, g) = \mathscr{T}_0(g) -\alpha$ where 
$\alpha\in Y_r$ is given by $\alpha_n = 
\dfrac{3}{4\pi} \kappa: u \otimes u 1_n$. 
Now, Corollary $\ref{ZER}$ ensures that the hypothesis 
of the implicit function theorem are satisfied. It provides
the existence part as well as the regularity part 
in theorem $\ref{TH}$.
\end{proof}
\begin{remark}
In the case $\epsilon = 0$ we have $ f_n \in L^r(S_2) $ \ for \ $ r \in [1, r_0[$. 
But we also know, from remark \ref{3:rem1}, that there exists $N \in \mathbb{N}^*$
such that $ f_n \in L^2(S_2) $ for $ n \geq N$.
Hence for $ \alpha \in \left[0, \frac 5 2 \right[ $ \ we have: 
\begin{align}
\| \sum_{n= 1}^\infty n^\alpha f_n(u) H_n(s) \|_{L^2(0, 1 ; L^r(S_2))} & \leq
\sum_{n= 1}^{N-1} n^\alpha \| f_n \|_{L^r(S_2)} + C
\| \sum_{n= N}^\infty n^\alpha f_n(u) H_n(s) \|_{L^2(0, 1 ; L^2(S_2))}
\nonumber
\\
& \leq C(N, \alpha) + C \left( \sum_{n= N}^{\infty} n^{2 \alpha} \| f_n \|_{L^2(S_2)}^2
\right)^{1/2} < + \infty
\nonumber
\end{align}
since $ \| f_n \|_{L^2(S_2)} \leq \frac c {n^3} $ \ for \ $ n \geq N$.
It follows that $ f \in H^{5/2 - \delta}(0, 1 ; L^r(S_2))$ for any $\delta > 0$ 
arbitrary small.

In contrast, we are unable to obtain such smoothness for $\epsilon \neq 0$. It comes from
the non linear term $ \epsilon \frac \partial {\partial s} (F k : \lambda(F))$
which couples the ''bad`` low frequencies with high frequencies.
\end{remark}
The main issue in the following is the non negativity of $F$, the other properties
could be obtained by rather simple means.
For instance the uniqueness can be proved by Holmgren's principle, but
we now argue differently.
%
\section{Some results on the evolution problem.}\label{5}
We prove in the sequel (sections $\ref{5}$ and $\ref{6}$) that the solution $F = f
+(4\pi)^{-1}$ 
obtained in theorem $\ref{TH}$ is the $L^1(Q)$ limit as $t\rightarrow +\infty$
of a family $\Big(\big(f^e+(4\pi)^{-1}\big)(t)\Big)_{t>0}$ of probability densities
which is solution of the corresponding evolution problem.
In the rest of this paper, we mostly restrict to 
exponent $r=1$,
in order to get uniqueness in the $L^1(S_2)$ frame. 
To begin 
with, consider the following evolution problem associated with equations 
$\eqref{EDPH}-\eqref{CLH}$:\\

Find $f^e(t, s, u)$ solution of $\eqref{EDD}, \eqref{CLL}, \eqref{CII}$: 
\\
\begin{align}
\frac{\partial f^e}{\partial t}
-\frac{\partial^2f^e}{\partial s^2}+  
\frac{\partial}{\partial u}\cdot(\mathcal{G}f^e)
-&\epsilon f^e\kappa : u\otimes u
+\epsilon \frac{\partial}{\partial s}\Big( f^e\kappa:\lambda(f)\Big)\nonumber\\
+&\dfrac{\epsilon}{4\pi}\int_{S_2}\kappa:v \otimes vf^e(s,v )d\mu(v)
=\dfrac{3+\epsilon}{4\pi}\kappa:u \otimes u  \textrm{ on } Q_T\label{EDD}\\
f^e(s=0) = f^e(s=1) = 0\phantom{g}\label{CLL}\\
f^e(t=0) = f_0^{e}\phantom{) = 0gllllfgg}\label{CII}
\end{align}
with $Q_T = [0, T] \times Q$, $T>0$. Function $f_0^{e}: Q \rightarrow 
\mathbb{R}$ is the initial data.
\\

Existence and uniqueness results for problem 
$(\ref{EDD}, \ref{CLL}, \ref{CII})$ have been obtained in 
\cite{CHP1}:

\begin{theorem}\label{IAL}
 Assume that $f_0^{e} \in L^2(Q)$ and  
$\dfrac{\partial f_0^e}{\partial u} 
\in  \big(L^2(Q)\big)^3$. Then, there 
exists a unique variational solution  
$f^{e} \in L^2\big(0, T, H_0^1(Q)\big)$ 
with 
$\dfrac{\partial f^e}{\partial t} 
\in L^2\big(0, T, H^{-1}(Q)\big)$ 
in the following sense:
\begin{align}
&- \int_{Q_T}f^e\dfrac{\partial \phi}{\partial t} dQ_T
- \int_{Q}f_0^e\phi(t=0) dQ
+
\int_{Q_T}
\Big[
\dfrac{\partial f^e}{\partial t} 
\dfrac{\partial \phi}{\partial s}
+
\dfrac{\partial }{\partial u}\cdot
\big(\mathcal{G}f^e\big)\phi\nonumber\\ 
&-\epsilon \kappa:u \otimes u f^e\phi
+\dfrac{\epsilon }{4\pi}
\int_{S_2}\kappa:v \otimes v f^edv \phi
-\epsilon f^e \kappa:\lambda(f^e)
\dfrac{\partial \phi}{\partial s}
\Big]dQ_T
=\dfrac{3+\epsilon }{4\pi}\int_{Q_T} 
\kappa:u \otimes u \phi \, dQ_T\label{BGX}
\end{align}
for any $\phi \in H^1\big(0, T:H_0^1(Q)\big)$
with $ \phi(t=T) = 0$. Moreover, 
if: $$f_0^e +\dfrac{1}{4\pi} \geq 0  \textrm{ a.e } 
(s, u) \in Q  \textrm{ and } \int_{S_2}
\Big(f_0^e +\dfrac{1}{4\pi}\Big) d\mu = 1 
\textrm{ a.e }   s \in ]0, 1[$$ then: 
$$f^e +\dfrac{1}{4\pi} \geq 0 \textrm{ a.e } 
(t, u, s) \in Q_T \textrm{ and } \int_{S_2}
\Big(f^e +\dfrac{1}{4\pi}\Big) d\mu = 1 
\textrm{ a.e }  (t, s) \in ]0, T[ \times ]0, 1[$$
\end{theorem}
From now on,
we assume that $f_0^{\epsilon}\in
H_0^1(Q)$ with
$f_0^e +\dfrac{1}{4\pi} \geq 0$
and 
$\int_{S_2}
\Big(f_0^e +\dfrac{1}{4\pi}\Big) d\mu = 1$.
From theorem 
$\ref{IAL}$, this implies the 
very useful estimate
(uniform 
in $t\in \mathbb{R}_+$): 
\begin{align}
\int_{S_2} \vert f^e \vert d\mu \leq 2 
\textrm{ a.e }  (t, s) \in \mathbb{R}_+ \times ]0, 1[ \label{CRUX}
\end{align}

As a consequence,we deduce from $\eqref{CRUX}$ that:
\begin{align}
\Vert \lambda (f^e)
\Vert_{L^{\infty}(0, T:W^{1, \infty}(0, 1))}
\leq 2\label{CRAX}
\end{align}
with $C\geq 0$ independent of $T \geq 0$.
We also have:
\begin{align}
\sum_{n=1}^{\infty}n^2\Vert f_{0, n}^{e}\Vert^2_{L^1(S_2)}\leq
C \sum_{n=1}^{\infty}n^2\Vert f_{0, n}^{e}\Vert^2_{L^2(S_2)}
<\infty \label{CRIX}
\end{align}
where
\begin{align}
f_{0, n}^{e} = \int_0^1
f_{0}^{e}(s) H_n(s)ds\label{CRLX}
\end{align}
for any $n\in\mathbb{N}^*$.  For $n\in \mathbb{N}^*$, let us denote,
$f_n^e = \int_0^1 f^e(s)H_n(s) ds \in L^2\big(0, T, H^1(S_2)\big)$. 
In $\eqref{BGX}$, taking $\phi(t, s, u) = \psi(t, u) H_n(s)$ with 
$\psi\in H^1\big(0, T, H^1(S_2)\big)$ and $\psi(t=T)=0$
as a test function, we easily obtain, for any $n\in\mathbb{N}^*$: 
\begin{align}
\dfrac{\partial f_n^e}{\partial t}+n^2\pi^2f_n^e
+\dfrac{\partial }{\partial u}\cdot
\big(\mathcal{G}f_n^e\big) 
-&
\epsilon \kappa u \otimes u f_n^e
+\dfrac{\epsilon }{4\pi}\int_{S_2}\kappa:v \otimes v
f_n^e(v) dv\nonumber\\
-&\sqrt{2}\epsilon n\pi \int_0^1 \big(f^e\kappa:\lambda(f^e)\big)(s)
cos(n\pi s)ds =\dfrac{3+\epsilon }{4\pi}\kappa:u\otimes u 1_n\label{EDPP}
\end{align}
All the terms appearing in the above equality belongs to 
$L^2\big(0, T, L^2(S_2)\big)$.

For the initial data, we have:
\begin{align}
f_n^e(t = 0) = f_{0, n }^e\label{CLPP}
\end{align}
where $f_{0, n }^e$ is given by $\eqref{CRLX}$.
For future reference note that: 
\begin{align}
f^e(t, s, u)  = \sum_{n=1}^{\infty}f_{n}^e(t, u)H_n(s)\label{DDJ}
\end{align} 
with convergence in 
$L^2\Big(0, 1,L^2\big(0, T, H^1(S_2)\big)\Big)$.
Last:
\begin{align}
\dfrac{\partial f_n^e}{\partial s}(s) =
\sum_{n=1}^{\infty} n\pi\sqrt{2}f_n^e(s)cos(n\pi s)\label{VART}
\end{align}
with convergence in $L^2(Q_T)$.

It is well known that for the 
heat equation, 
estimates of two derivatives with respect to the space variables can be obtained
in suitable spaces. 
From that point of view, 
estimates with respect to the $s$ derivatives in theorem 
$\ref{IAL}$ do not seem to 
be optimal. The following simple estimate
will be enough for our purposes:
\begin{lemma}\label{SOR}
With the notations, and under the hypothesis 
of theorem $\ref{IAL}$, there exist $\epsilon_0>0$
such that: 
\begin{align}
\sum_{n=1}^{\infty} \int_0^T
n^{4}\Vert f_n^e(t) \Vert_{L^1(S_2)}^2 dt
\leq C(T) < \infty\label{ZAZA}
\end{align}
for any $T\geq 0$ and $\epsilon\in ]0, \epsilon_0[$
\end{lemma}
\begin{proof}
Let us denote $g^e =  \dfrac{\partial }{\partial s}\Big[
f^e\kappa:\lambda(f^e)\Big]$. Function $g^e$ is an element of 
$L^2(Q_T)$ due to $\eqref{CRAX}$. For $n\in\mathbb{N}^*$, we write 
as usual
$g_n^e = \int_0^1 g^e(s) H_n(s)ds$. Function $g_n^e$ belongs to 
$L^2(]0, T[ \times S_2)$, and by the  Holder and Bessel inequalities:
\begin{align}
\sum_{n=1}^{\infty} 
\Vert g_n^e(t) \Vert_{L^1(]0, T[\times S_2)}^2
\leq
C\sum_{n=1}^{\infty} 
\Vert g_n^e(t) \Vert_{L^2(]0, T[\times S_2)}^2 
\leq C\Vert g^e(t) \Vert_{L^2(Q_T)}^2\label{LOUF}
\end{align}
We multiply $\eqref{EDPP}$
by $sgn(f_n^e)$
and integrate on $S_2$. It gives, for $n\in\mathbb{N}^*$:
\begin{align}
\dfrac{d }{dt}\Vert f_n^e \Vert_{L^1(S_2)} 
+
n^2\pi^2 \Vert f_n^e \Vert_{L^1(S_2)} 
\leq 
C\epsilon\Vert f_n^e \Vert_{L^1(S_2)} 
+
\Vert g_n^e \Vert_{L^1(S_2)}
+C\vert 1_n\vert\label{QZAI}
\end{align}
In the above inequality, we have used 
lemma $\ref{POK}$ with $r=1$
and identity 
$\int_{S_2} \dfrac{\partial }{\partial u}\cdot
\big(\mathcal{G}\vert f_n^e \vert\big)d\mu = 0$.
Remark that $\Vert f_n^e \Vert_{L^1(S_2)}$ belongs to $H^1(0, T)$.
Now, we fix $m\in\mathbb{N}^*$, multiply $\eqref{QZAI}$ 
by $n^2\Vert f_n^e \Vert_{L^1(S_2)}$ and take the 
sum from $n=1$ to $m$. Using the fact that: 
\begin{align}
\Big(\Vert g_n^e \Vert_{L^1(S_2)}
+
C\vert 1_n \vert \Big)
n^2\pi^2\Vert f_n^e \Vert_{L^1(S_2)}
\leq
\dfrac{\pi^4 n^4}{4}
\Vert f_n^e \Vert_{L^1(S_2)}^2
+
2\Vert g_n^e \Vert_{L^1(S_2)}^2
+
2C^2\vert 1_n \vert^2
\end{align}
we deduce from 
$\eqref{LOUF}$, $\eqref{UGH}$ and $\eqref{QZAI}$ 
that:
\begin{align}
\dfrac{d }{dt}\Big(\sum_{n=1}^m n^2 \Vert f_n^e \Vert_{L^1(S_2)}^2 \Big)
+
\dfrac{\pi^2}{2}
\sum_{n=1}^m
n^4 \Vert f_n^e \Vert_{L^1 (S_2)}^2
\leq 
C\label{QMAY}
\end{align}
with $C>0$ independent of $m\in\mathbb{N}^*$.
Integrating this inequality with respect to $t$ and appealing to
$\eqref{CRIX}$, we obtain the result.
\end{proof} 
\section{Proof of uniqueness.
Function $F$ is a probability density.}\label{6}
We denote by $f^d =f^e-f$. As before, function
$f$ is a stationary solution constructed in section
$\ref{4}$ and $f^e$ is the solution of the evolutionary
Doi-Edwards equation
(see section $\ref{5}$). We now prove that 
$f^d(t) \rightarrow 0$ in a suitable norm when 
$t\rightarrow + \infty$. This will provide at the same 
time uniqueness of $f$ and the fact that 
$f+(4\pi)^{-1}$ is a probability density.
Notice that  long time behavior of 
some systems 
arising in the theory of polymeric fluids
(Hookean model and FENE model) 
are studied for instance in \cite{JLB2}.

Since 
\begin{align}
f^e\in H^1\Big(0, 1, L^2\big(]0, T[ \times S_2\big)\Big)\label{DOMB}
\end{align}
and 
\begin{align}
f\in W^{1, \infty}\big(0, 1, L^r(S_2)\big) \textrm{ with }r\in[1, r_0]\label{DIMB}
\end{align}
we have:
\begin{align}
f^d\in H^1\Big(0, 1, L^2\big(0, T, L^r(S_2)\big)\Big)\label{DYMB}
\end{align}
We also have:
\begin{align}
f^d(s) = \sum_{n=1}^{\infty} f_n^d H_n(s)\label{TVX}
\end{align}
with convergence in $H^1\Big(0, 1, L^2\big(0, T, L^r(S_2) \big)\Big)$, 
with $f_n^d = f_n^e-f_n$ (see $\eqref{DDJ}$, $\eqref{VART}$, and lemma $\ref{FOU}$).
Remark that from lemma $\ref{SOR}$ and the fact 
that $f\in X_1$ we have: 
$$\sum_{n=1}^{\infty} \int_0^T
n^{4}\Vert f_n^d(t) \Vert_{L^1(S_2)}^2 dt
\leq C(T) < \infty.$$
Now, from $\eqref{EDPP}$ and equality
$\mathscr{T}(\epsilon, f) = 0$, we obtain:
\begin{align}
\dfrac{\partial f_n^d}{\partial t}+n^2\pi^2f_n^d
+&\dfrac{\partial }{\partial u}\cdot
\big(\mathcal{G}f_n^d\big) 
-
\epsilon \kappa u \otimes u f_n^d
+\dfrac{\epsilon }{4\pi}\int_{S_2}\kappa:v \otimes v
f_n^d(v) dv\nonumber\\
+&\epsilon
\int_0^1 
\dfrac{\partial}{\partial s}
\big[\kappa:\lambda(f^d)f\big](s)
H_n(s)ds
+\epsilon
\int_0^1 
\dfrac{\partial}{\partial s}
\big[\kappa:\lambda(f^e)f^d\big](s)
H_n(s)ds=0\label{EDAD}
\end{align}
for any $n \in \mathbb{N}^*$. 
Remark that:
\begin{align}
I_{1n} := \int_0^1 \dfrac{\partial}{\partial s}
\big[\kappa:\lambda(f^e)f^d\big] H_n(s) \, ds
\in L^{2}\big(]0, T[
, L^r(S_2)\big)\label{TYTY}
\end{align}
and:
\begin{align}
I_{2n}:= \int_0^1 \dfrac{\partial}{\partial s}
\big[\kappa:\lambda(f^d)f\big]
H_n(s) \, ds
\in L^{\infty}\big(]0, T[
, L^r(S_2)\big)\label{KYLY}
\end{align}
due to $\eqref{DOMB}$, $\eqref{DIMB}$, $\eqref{DYMB}$
and $\eqref{CRAX}$. Notice also that $\dfrac{\partial }{\partial t}
\Vert f_n^d \Vert_{L^{1}(S_2)} \in L^2(0, T)$. We multiply 
$\eqref{EDAD}$ by $sgn(f_n^d)$ and integrate 
on $S_2$. Using lemma 
$\ref{POK}$ with $r=1$, we obtain: 

\begin{align}
\dfrac{\partial }{\partial t}\Vert f_n^d \Vert_{L^{1}(S_2)}
+n^2\pi^2 \Vert f_n^d \Vert_{L^{1}(S_2)}
\leq
2\epsilon\int_{S_2}\vert  \kappa :u \otimes u f_n^d \vert d\mu
+
\epsilon \Vert I_{1n} \Vert_{L^{1}(S_2)}+\epsilon
\Vert I_{2n} \Vert_{L^{1}(S_2)}
\label{EDDD}
\end{align}
The goal is now to
multiply $\ref{EDDD}$ by 
$n^2\Vert f_n^d \Vert_{L^{1}(S_2)}$ 
and take the sum 
from $n=1$ to $\infty$. We will need some
preliminary lemmas.
\begin{lemma}\label{YTRE}
\begin{align}
\sum_{n=1}^{\infty}n^2
\Vert f_n^d \Vert_{L^{1}(S_2)}
\Vert I_{1n} \Vert_{L^{1}(S_2)}
\leq 
C\sum_{n=1}^{\infty}n^4
\Vert f_n^d \Vert_{L^{1}(S_2)}^2
\end{align}
with $C>0$ independent of $t$.
\end{lemma}
\begin{proof}
We have:
\begin{align}
I_{1n}(t, u) = 
\sum _{q=1}^{\infty}
a_{qn}(t)f_q^d(t, u)\label{XXW}
\end{align}
where $a_{qn}(t) = \int_0^1
\dfrac{\partial}{\partial s}
\big[\kappa:\lambda(f^e)H_q(s)\big]
H_n(s)ds$. The right hand side of 
$\eqref{XXW}$ is convergent in $L^2\big(0, T, L^1(S_2)\big)$ 
due to the convergence in
$H^1\Big( 0, 1,  L^2\big(0, T, L^1(S_2)\big)\Big)$
of 
$\sum _{q=1}^{\infty}
f_q^dH_q$ .
From $\eqref{XXW}$, we deduce:
\begin{align}
\Vert I_{1n}(t)\Vert_{L^1(S_2)} \leq
\sum _{q=1}^{\infty}
\vert a_{qn}(t)\vert \Vert f_q^d(t) \Vert_{L^1(S_2)}\label{BLAD}
\end{align}
Now, we observe that, due to $\eqref{CRAX}$:
\begin{align}
\sum_{n=1}^{\infty} \vert a_{qn}\vert^2
=
\Vert
\dfrac{\partial}{\partial s}
\Big(
\kappa:\lambda(f^e) H_q \Big)
\Vert_{L^2(0, 1)}^2
\leq C q^2\label{BLED}
\end{align}
with $C > 0$ independent of $t$.
We deduce from $\eqref{BLAD}$
that
\begin{align}
\sum_{n=1}^{\infty}
n^2
\Vert f_n^d(t) \Vert_{L^1(S_2)}
\Vert I_{1n} \Vert_{L^1(S_2)}
\leq &
\sum_{n=1}^{\infty}\sum_{q=1}^{\infty}
n^2
\vert a_{qn} \vert
\Vert f_n^d \Vert_{L^1(S_2)}
\Vert f_q^d \Vert_{L^1(S_2)}\nonumber\\
\leq &
\sum_{n=1}^{\infty}\sum_{q=1}^{\infty}
\dfrac{\vert a_{qn} \vert}{q^2}
n^2
\Vert f_n^d \Vert_{L^1(S_2)}
q^2 \Vert f_q^d \Vert_{L^1(S_2)}\label{SDF}
\end{align}
Using Cauchy-Schwarz inequality, we obtain:
\begin{align}
\sum_{n=1}^{\infty}
n^2
\Vert f_n^d(t) \Vert_{L^1(S_2)}
\Vert I_{1n} \Vert_{L^1(S_2)}
\leq 
\Big(
\sum_{n=1}^{\infty}\sum_{q=1}^{\infty}
\dfrac{\vert a_{qn} \vert^2}{q^4}
\Big)^{1/2}
\Big(
\sum_{n=1}^{\infty}
n^4
\Vert f_n^d \Vert_{L^1(S_2)}^2
\Big)
\end{align}
With $\eqref{BLED}$, this gives the result.
\end{proof}
Before giving bounds on  
$\Vert I_{2n} \Vert_{L^1(S_2)}$, we establish
the following:

\begin{lemma}\label{FLOM}
There exists $C>0$ such that for any 
$N\in\mathbb{N}^*$, we have:
\begin{align}
\Vert
\int_0^1 f(s)cos(N\pi s) ds
\Vert_{L^1(S_2)}
\leq
\dfrac{C}{N^2} \Vert f \Vert_{X_1}\label{CPAMOI}
\end{align}
\end{lemma}
\begin{proof}
We have:
\begin{align}
\int_0^1 f(s)cos(N\pi s) ds
&= 
\sqrt{2} \sum_{p=1}^{\infty}
f_p 
\int_0^1 sin(p\pi s)cos(N\pi s)ds
\nonumber\\
&= 
\sqrt{2}
\sum_{p \neq N}
\dfrac{p}{p^2-N^2}\big[ 1+(-1)^{p+N+1}\big]f_p\label{KXC}
\end{align}
Now, from definition of 
$\Vert f \Vert_{X_r}$ 
we have:
\begin{align}
\Vert
\int_0^1 f(s)cos(N\pi s) ds
\Vert_{L^1(S_2)}
\leq &
2\sqrt{2} 
\Vert f \Vert_{X_1}
\sum_{p \neq N}^{ }
\dfrac{1}{p^2 \vert p^2-N^2 \vert}\label{MMF}
\end{align}
Next, remark that:
\begin{align}
\sum_{p \neq N}^{ }
&\dfrac{1}{p^2 \vert p^2-N^2 \vert}\nonumber\\
=&
\sum_{1 \leq |p-N| < N/2}
\dfrac{1}{p^2}
\dfrac{1}{ \vert p-N \vert \vert p+N \vert}
+
\sum_{1 \leq p \leq (N/2)}^{ }
\dfrac{1}{p^2 \vert p^2-N^2 \vert}
+
\sum_{ p \geq (3N/2)}^{ }
\dfrac{1}{p^2 \vert p^2-N^2 \vert}
\nonumber\\
\leq &
\Big(\dfrac{2}N \Big)^2
\Big(\sum_{N/2 < p < 3N/2}^{ }
\dfrac{1}{N}\Big)
+
\dfrac{4}{3N^2}
\Big(\sum_{p=1}^{\infty}
\dfrac{1}{p^2}\Big)
+
\dfrac{4}{5N^2}
\Big(\sum_{p=1}^{\infty}
\dfrac{1}{p^2}\Big)
\nonumber\\
\leq & 
\dfrac{C}{N^2}\nonumber
\end{align}
Together with $\eqref{KXC}$, $\eqref{MMF}$, this ends the proof.
\end{proof}
As a consequence, we have the following 
estimate on the nonlinear term:
\begin{lemma}\label{BRT}
There exists $C>0$ such that for any 
$q \in \mathbb{N}^*$, $n \in \mathbb{N}^*$
we have:
\begin{align}
\Vert
\int_0^1 
\dfrac{\partial }{\partial s}
\Big[\int_0^s
H_q(\tau)d\tau
f(s,u)\Big]
H_n(s)ds
\Vert_{L^1(S_2)}
\leq
C\dfrac{q}{n} \Vert f \Vert_{X_1}\label{6:CPAMOI}
\end{align}
\end{lemma}
\begin{proof}
Since $\int_0^s
H_q(\tau)d\tau = \dfrac{\sqrt{2}}{q\pi}
\big[ 1- cos(q\pi s)\big]$, we obtain, integrating by parts: 
\begin{align}
\int_0^1 
\dfrac{\partial }{\partial s}
\Big[\int_0^s
H_q(\tau)d\tau
f(s,u)\Big]
H_n(s)ds = 
\dfrac{n}{q} (E_1+ E_2 + E_3)\label{VYT}
\end{align}
where:
\begin{itemize}
\item $E_1 = - 2 \int_0^1 f(s) cos(n\pi s)ds$
\item $E_2 = \int_0^1 f(s) cos\big((n+q)\pi s\big)ds$
\item $E_3 = \int_0^1 f(s) cos\big((n-q)\pi s\big)ds$
\end{itemize}
Using lemma $\ref{FLOM}$, we have:
\begin{align}
\Vert E_j \Vert_{L^{1}(S_2)} \leq \dfrac{C}{n^2}
\Vert f \Vert_{X_1},  \textrm{ for } j=1, 2 \label{KLOP}
\end{align}
and also:
\begin{align}
&\bullet \Vert E_3 \Vert_{L^{1}(S_2)} \leq C
\Vert f \Vert_{X_1},  \textrm{ for } q=n \label{KLIP}\\
&\bullet\Vert E_3 \Vert_{L^{1}(S_2)} \leq \dfrac{C}{\vert n-q \vert^2}
\Vert f \Vert_{X_1},  \textrm{ for } q \neq n\nonumber
\end{align}
Now, notice that:
\begin{align}
\dfrac{1}{\vert n-q \vert}
= 
\dfrac{1}{n}
\Big\vert
1+
\dfrac{q}{n-q}
\Big\vert
\leq 2
\dfrac{q}{n} \quad \textrm{for} \quad q\neq n
\end{align}
Hence:
\begin{align}
\Vert E_3 \Vert_{L^{1}(S_2)}  \leq 
C \dfrac{q^2}{n^2} \Vert f \Vert_{X_1} 
\quad \forall q, n \in \mathbb{N}^*.
\label{ROS}
\end{align}
From 
$\eqref{VYT}$, $\eqref{KLOP}$, $\eqref{ROS}$, we get
the result.
\end{proof}
We deduce from lemma 
$\ref{BRT}$ the required estimate 
on $\Vert I_{2n} \Vert_{L^{r}(S_2)}$:
\begin{lemma}\label{TRN}
$$\sum_{n=1}^{\infty}
n^2 \Vert f_n^d \Vert_{L^{1}(S_2)}
\Vert I_{2n} \Vert_{L^{1}(S_2)}
\leq  C 
\sum_{n=1}^{\infty}
n^4\Vert f_n^d \Vert_{L^{1}(S_2)}^2$$
with $C>0$ independent of $t$.
\end{lemma}
\begin{proof}
We can write
$$
I_{2n} 
=
\sum_{q=1}^{\infty} \kappa: \tilde{\lambda}(f_q^d)
\int_0^1 
\dfrac{\partial}{\partial s} 
\Big[
\int_0^s H_q(\tau)d\tau f(s)
\Big]
H_n(s)ds
$$
where we denote
$$
\tilde{\lambda}(f_q^d) = \int_{S_2} f_q^d (t,v) v \otimes v \, dv.
   $$
Using lemma $\ref{BRT}$, we deduce:
$$
\Vert I_{2n} \Vert_{L^{1}(S_2)}
\leq 
\sum_{q=1}^{\infty} 
\Vert f_q^d \Vert_{L^{1}(S_2)}
\dfrac{q}{n}
\Vert f \Vert_{X_1}
$$
which gives:
\begin{align}
\sum_{n=1}^{\infty}
n^2 \Vert f_n^d \Vert_{L^{1}(S_2)}
\Vert I_{2n} \Vert_{L^{1}(S_2)}
\leq &
\sum_{q=1}^{\infty}
\sum_{n=1}^{\infty}
nq
\Vert f_n^d \Vert_{L^{1}(S_2)}
\Vert f_q^d \Vert_{L^{1}(S_2)}
\Vert f \Vert_{X_1}\nonumber\\
= &
\Big[
\sum_{n=1}^{\infty}
n
\Vert f_n^d \Vert_{L^{1}(S_2)}
\Big]^2\Vert f \Vert_{X_1}\label{WAP}
\end{align}
From Cauchy-Schwarz inequality, we 
get:
\begin{align}
\sum_{n=1}^{\infty}
n
\Vert f_n^d \Vert_{L^{1}(S_2)}
\leq 
\Big[\sum_{n=1}^{\infty}
\dfrac{1}{n^2}
\Big]^{1/2}
\Big[\sum_{n=1}^{\infty}
n^4
\Vert f_n^d \Vert_{L^{1}(S_2)}^2
\Big]^{1/2}\label{WIP}
\end{align}
Inequalities $\eqref{WAP}$ and $\eqref{WIP}$
provides the result.
\end{proof}
We finally give the proof of the last part 
of theorem $\ref{TH}$.

\begin{proof}.
We multiply $\eqref{EDDD}$
by $n^2\Vert f_n^d \Vert_{L^{1}(S_2)}$
and take the sum from 
$n=1$ to $m\in\mathbb{N}^*$.
For $|\epsilon|$ small 
enough, and using lemma $\ref{YTRE}$
and lemma $\ref{TRN}$, we obtain:
\begin{align}
\dfrac{d}{dt} 
\xi_m(\tau)
+
\chi_m(\tau)
\leq 
C |\epsilon|
\chi(\tau)
\label{VQI}
\end{align}
where we denote
\begin{align}
& \xi_m(\tau) = \sum_{n=1}^{m} n^2
\Vert f_n^d(\tau)\Vert_{L^{1}(S_2)}^2
\nonumber
\\
& \chi_m(\tau) = \sum_{n=1}^{m} n^4
\Vert f_n^d(\tau)\Vert_{L^{1}(S_2)}^2
\nonumber
\\
& \xi(\tau) = \sum_{n=1}^{\infty} n^2
\Vert f_n^d(\tau)\Vert_{L^{1}(S_2)}^2
\nonumber
\\
& \chi(\tau) = \sum_{n=1}^{\infty} n^4
\Vert f_n^d(\tau)\Vert_{L^{1}(S_2)}^2
\nonumber
\end{align}
Let $t\in[0, T]$. We multiply $\eqref{VQI}$ by 
$e^{\tau/2}$, integrate from $\tau=0$ to 
$\tau=t$ and we get: 
\begin{align}
\xi_m(t) e^{t/2} - \frac 1 2 \int_0^t \xi_m(\tau) e^{\tau/2} d \tau +
\int_0^t \chi_m(\tau) e^{\tau/2} d \tau
\leq C |\epsilon| 
\int_0^t \chi(\tau) e^{\tau/2} d \tau + \xi_m(0).
\label{FLEB}
\end{align}
%
Now we pass to the limit $m\rightarrow +\infty$ in 
$\eqref{FLEB}$, which gives for $|\epsilon|$ small enough
$$
\xi(t) e^{t/2} - \frac 1 2 \int_0^t \xi(\tau) e^{\tau/2} d \tau +
\frac 1 2 \int_0^t \chi(\tau) e^{\tau/2} d \tau
\leq C.
  $$
Since $\chi(t) \geq \xi(t)$ we deduce 
$ \xi(t) \leq C e^{-t/2} $ which
implies that $\xi(t) \rightarrow 0$ 
when $t \rightarrow +\infty$.
Since by the  Cauchy-Schwarz inequality and the definition 
of function $\xi$ we have: 
\begin{align}
\sum_{n=1}^{\infty}
\Vert f_n^d(t)\Vert_{L^{1}(S_2)}
\leq &
C 
\Big[\sum_{n=1}^{\infty}
\dfrac{1}{n^2}\Big]^{1/2}
\Big[\sum_{n=1}^{\infty}
n^2
\Vert f_n^d(t)\Vert_{L^{1}(S_2)}^2\Big]^{1/2}\nonumber\\
\leq &
C \sqrt{\xi (t)}\label{TRALALA}
\end{align}
we finally get (see $\eqref{TVX}$): 
\begin{align}
\Vert f^d (t)\Vert_{L^{\infty}(0, 1, L^1(S_2))}
\rightarrow 0 \textrm{ when } t \rightarrow +\infty\label{TRALALO}
\end{align}
As a consequence,
$$\Big\Vert \int_{S_2} f^d(t) d\mu \Big\Vert_{L^{\infty}(0, 1)}
\rightarrow 0
\textrm{ when } t \rightarrow +\infty
$$ 
Recall that
$\int_{S_2} f^e(t) d\mu = 1$. Hence $\int_{S_2} f d\mu = 1$
for almost every $s\in ]0, 1[$, and in fact for every 
$s\in [0, 1]$ due to $\eqref{DIMB}$ and Sobolev embeddings.
The uniqueness follows also from $\eqref{TRALALO}$
since for  another solution $g$ of $\eqref{EDPH}-\eqref{CLH}$, we have 
that $\Vert f-g \Vert_{L^{\infty}(0, 1, L^1(S_2))}
\rightarrow 0$ when $t \rightarrow +\infty$, hence $f=g$.

It remains to prove the non negativity of $f + \frac 1 {4 \pi}$.
To do that, let us consider an arbitrary function $\varphi \in C(S_2)$ with
$\varphi \geq 0$ on $S_2$. We have
\begin{equation}
\int_{S_2} \Big( f + \frac 1 {4 \pi} \Big) \varphi \, d \mu = 
\int_{S_2} \Big( f^e + \frac 1 {4 \pi} \Big) \varphi \, d \mu
- \int_{S_2} f^d \varphi \, d \mu
\end{equation}
Since $f^e + \frac 1 {4 \pi} \geq 0$ we obtain with the help of $\eqref{TRALALO}$ that
$\int_{S_2} \Big( f + \frac 1 {4 \pi} \Big) \varphi \, d \mu \geq 0$
for almost every $s\in ]0, 1[$, and in fact for every 
$s\in [0, 1]$. 
This completes the proof.

\end{proof}

\section{Bibliography} 


\end{document}